\documentclass[12pt]{amsart}
\usepackage{amsmath,amstext,amsfonts,amsthm,amssymb,epigraph}
\usepackage{eucal}
\usepackage{framed}
\usepackage{xcolor}
\usepackage{bbm}
\usepackage[all]{xy}                                        %
\swapnumbers
\newtheorem{thm}[subsubsection]{Theorem}
\newtheorem{lem}[subsubsection]{Lemma}
\newtheorem{prp}[subsubsection]{Proposition}

\newtheorem*{Lem}{Lemma}

{  \theoremstyle{definition}
           \newtheorem{dfn}[subsubsection]{Definition}
           \newtheorem{exm}[subsubsection]{Example}
           \newtheorem{rem}[subsubsection]{Remark}
           
           \newtheorem{notation}[subsubsection]{Notation}

}

\newcommand{\Free}{\mathit{Free}}

\newcommand{\id}{\mathrm{id}}

\renewcommand{\int}{\mathit{int}}

\mathchardef\mhyphen="2D

\newcommand{\Sk}{\mathtt{Sk}} 

\newcommand{\Span}{\operatorname{Span}}

\newcommand{\bR}{\mathbb{R}}
\newcommand{\bV}{\mathbb{V}}
\newcommand{\bW}{\mathbb{W}}

\newcommand{\cP}{\mathcal{P}}

\newcommand{\cR}{\mathcal{R}}

\newcommand{\fg}{\mathfrak{g}}
\newcommand{\fh}{\mathfrak{h}}

\newcommand{\Z}{\mathbb{Z}}

\newcommand{\R}{\mathbb{R}}

\begin{document}

\title[]{Matsumoto theorem for skeleta}
\author{M.~Gorelik}
\address{Weizmann Institute of Science}
\email{maria.gorelik@gmail.com}
\author{V.~Hinich}
\address{University of Haifa}
\email{vhinich@gmail.com}
\author{V. Serganova}
\address{UC Berkeley}
\email{serganov@math.berkeley.edu}
 
\begin{abstract} 
We present a proof of  a generalization of the theorem of
H.~Matsumoto on Coxeter groups. Our generalized version is applicable to
``graphs admitting geometric realization''. The original version of the theorem for Coxeter groups is a special case when applied to the Cayley graph and the geometric representation of a Coxeter group.
Our version of Matsumoto theorem is also applicable to 
{\sl skeleta}, graphs that were defined in the recent paper~\cite{GHS} 
on root Lie superalgebras.
\end{abstract}
\maketitle

\section{Introduction}

\subsection{}
Let $(G,S)$ be a Coxeter group.
Recall that this means that 
$$
G=\Free_S/((st)^{m_{st}}),
$$
the factor of the free group generated by $S$, where $M=(m_{st})$
is a symmetric matrix over $\Z$ satisfying the conditions
$m_{ss}=1,\ m_{st}>1$ for $s\ne t$.
A reduced decomposition of $g\in G$ is a shortest expression
$$
g=s_1\ldots s_n, \ s_i\in S
$$
and we denote by $\ell(g)$ the length of a reduced decomposition of $g$.

It is convenient to rewrite the relations as follows
$$
\begin{cases}
s^2=1, & s\in S \\
sts\ldots =tst\ldots \text{ (both expressions of length }m_{st}),&
  s\ne t\in S.
\end{cases}
$$
 
The second equality is called the generalized braid (also Artin-Tits) relation.  
We define an equivalence relation (called {\sl the braid relation})
on the set of words in the alphabet $S$
as the minimal relation closed under concatenation and containing
the Artin-Tits relation.

The following result is proven by Matsumoto~\cite{M} in 1964.

\begin{thm}
\label{thm:matsumoto-0}
Let $g=s_1\ldots s_n=t_1\ldots t_n$ be two reduced decompositions of $g$,
with $s_i$ and $t_j$ in $S$. Then the sequences $(s_1,\ldots,s_n)$
and $(t_1,\ldots,t_n)$ are braid equivalent.
 \end{thm}

The original proof of Matsumoto theorem is based on the following
{\sl replacement property}.

\begin{Lem}
Let $g=s_1\ldots,s_n$ be a reduced decomposition and let $s\in S$
satisfy the condition
$$
\ell(sg)<n.
$$
Then there exists $j$ such that
$$
ss_1\ldots s_n=s_1\ldots s_{j-1}s_{j+1}\ldots s_n.
$$
\end{Lem}

\subsection{}
In a recent paper~\cite{GHS} the authors introduced a family of Lie superalgebras (root Lie superalgebras)
generalizing Kac-Moody superalgebras and Borcherds Lie algebras. 
The set of Borel subalgebras of a root Lie superalgebra can be described 
in terms of a certain graph called a skeleton. The skeleton plays the role 
of the Cayley graph of a Weyl group. A version of Matsumoto theorem 
can be formulated for these graphs but the 
replacement property does not seem to hold. This led us to an 
attempt to generalize Matsumoto theorem so that it would cover the case
of skeleta. 

Our solution turns out to be very simple: it gives a proof that is,
in our opinion, easier even for the Coxeter groups. Note that our proof
does not use replacement property. Our approach is somewhat similar (though not completely equivalent) to that of 
Heckenberger-Yamane~\cite{HY}.
It seems that validity of the Matsumoto theorem for skeleta in the
fully reflectable case could have been deduced from ~\cite{HY} and 
the Coxeter property of skeleta proven in~\cite{GHS}, Sect.~6. 

Matsumoto graphs have a dual presentation in terms of convex geometry of
hyperplanes arrangements, see~\ref{ss:dual}. In the case of finite 
crystallographic arrangements a connection to Weyl groupoids was studied by M.~Cuntz \cite{Cu}.

\section{Matsumoto graphs}

\subsection{Rays and cones}
Our basic geometric object will consist of closed rays in a real vector space 
$\bV$.

For a collection of rays $\underline\alpha=(\alpha_1,\ldots,\alpha_k)$ their sum
$C(\underline\alpha)=C(\alpha_1,\ldots,\alpha_k)$ is the convex cone generated by $\alpha_i$. For a
collection $\underline\alpha=(\alpha_i,\ i=1,\ldots,\dim\bV)$ of linearly independent rays
the cone $C(\underline\alpha)$ will be called the {\sl simplicial cone}
spanned by $\alpha_i$.

Note that the rays spanning a simplicial cone $C$ are uniquely 
determined by $C$.

\subsection{Graphs and their realization}
Our graphs $\Gamma=(V,E)$ are connected. 
Their edges may be compact (having two ends) or noncompact
(infinite, connected to only one vertex), and we write $E=E_c\sqcup E_\infty$.

We think of compact edges as the 
ones allowing both orientations, whereas infinite edges are always incoming. We define the set of oriented edges
$\stackrel{\to}{E}=\stackrel{\to}{E_c}\sqcup E_\infty$ where 
$\stackrel{\to}{E_c}$ is a two-fold covering of $E_c$ that we identify
with $E_c\times\{\pm 1\}$. Thus, 
$$
\stackrel{\to}{E}=\{(e,\epsilon)\in E\times\{\pm 1\}|\ \epsilon=1
\mathrm{\ if\ }e\in E_\infty\}.
$$

A geometric realization of a graph $\Gamma$ is an assignment
of a ray $\alpha_{e,\epsilon}$ to each oriented edge $(e,\epsilon)$ so that
$\alpha_{e,-\epsilon}=-\alpha_{e,\epsilon}$ for $e\in E_c$. 
If $e\in E_\infty$, we will write $\alpha_e$ instead of $\alpha_{e,1}$.

We define the set of roots $R=\{\alpha_{e,\epsilon}\}$ and we require
the following axioms.
\begin{itemize}
\item[1.] If $e\in E_\infty$  then $-\alpha_e\not\in R$.
\item[2.] For any $v\in V$ the set of roots $\alpha_1,\ldots,\alpha_d$ 
assigned to the edges ingoing to $v$, form a basis of $\bV$. In particular, the graph $\Gamma$ is regular of degree $d=\dim\bV$.
We denote $R^+_v=\{\alpha\in R|\ \alpha\in C(\alpha_1,\ldots,\alpha_d)\}$.

\item[3.] For each $(e,\epsilon):v\to v'$ one has
$$
R^+_{v'}\setminus R^+_v=\{\alpha_{e,\epsilon}\}.
$$
Note that this condition implies that a graph having  a geometric realization has no loops $e:v\to v$ and no multiple edges.
\item[4.] If $R^+_v=R^+_{v'}$ then $v=v'$.
\end{itemize}

\begin{dfn}
A Matsumoto graph is a connected graph endowed with a geometric realization.
\end{dfn}

Note that a root $\alpha$ corresponds to a compact edge iff $-\alpha\not\in R$. Having this in mind,  the roots
corresponding to compact edges will be called {\sl invertible} and the
roots of infinite edges will be called {\sl non-invertible}.

Our generalization of Matsumoto theorem is formulated in terms of
Matsumoto graphs. It claims that any two shortest paths from one vertex
to another in a Matsumoto graph are braid equivalent. The result is
formulated only in~\ref{thm:matsumoto} as the braid equivalence requires
a certain effort to define. 

\subsection{Basic properties}

The property below immediately follows from the axioms.
\begin{lem}
\begin{itemize}
\item[1.] For any $e\in E_\infty$ and any $v\in V$ one has $\alpha_e\in R^+_v$.
\item[2.] $R\subset R^+_v\sqcup(-R^+_v)$. 

\end{itemize}
\end{lem}\qed

We define $R^-_v:=R\setminus R^+_v\subset -R^+_v$.
\begin{lem}
\label{lem:positives-differ}
If $R^+_v\subset R^+_{v'}$ then $v=v'$.
\end{lem}
\begin{proof}
Since the graph is connected, the difference between any pair of sets $R^+_v$ and $R^+_{v'}$ is finite.
Therefore, we can talk about relative cardinality of these sets.
A passage from a vertex to its neighbor along an edge does not change
the cardinality. Therefore, the relative cardinality is always zero.
Thus, by the assumption $R^+_v=R^+_{v'}$ and by Axiom 4 $v=v'$.
\end{proof}

For vertices $v,v'$ we define the distance $d(v,v')$ as the length of a shortest path connecting $v$ with $v'$. 

\begin{lem}
\label{lem:short}
Any path from $v$ to $v'$ has length equal to $d(v,v')$ modulo 2.
$d(v,v')$ is the cardinality of $R^+_{v'}\setminus R^+_v$.
In particular, any closed path has an even cardinality; that is,
any graph admitting a geometric realization, is bipartite.
\end{lem}
\begin{proof}
The first claim follows from Axiom 3.
The inequality  $d(v,v')\geq|R^+_{v'}\setminus R^+_v |$ is obvious. The
converse is deduced by induction, as by~\ref{lem:positives-differ}
for $v'\ne v$ there exists
an incoming edge $(e,\epsilon)$
to $v$  which does not belong to $R^+_{v'}$. 
\end{proof}

\subsubsection{}
Matsumoto graphs are regular of degree $d=\dim\bV$. The lemma below shows that any edge $e$ connecting two vertices $v$ and 
$w$ establishes a one-to-one correspondence between the edges adjacent to
$v$ and to $w$ so that $e$ corresponds to itself. In Proposition~\ref{prp:coloring}
we will show that this correspondence leads to a consistent coloring of the set
of edges $E$ with the elements of a set $X$ of cardinality $d$, so that the correspondence above assigns to an edge adjacent to $v$ the edge
adjacent to $w$ of the same color.

\begin{lem}
\label{lem:color}
Let $(e,\epsilon):v\to v'$, $\alpha=\alpha_{e,\epsilon}$  and $R^+_v=C(\alpha_1,\ldots,\alpha_d)$
with $\alpha_1=-\alpha$. 
Then there is a unique numbering of the basic roots
of $R^+_{v'}=C(\alpha'_1,\ldots,\alpha'_d)$  such that 
$\alpha'_1=-\alpha_1=\alpha$ and for each $i>1$ 
$\alpha'_i$ lies in the cone spanned by $\alpha_i$ and $\alpha_1$.
\end{lem}
\begin{proof} Let $\bV'=\bV/\bR\alpha$ and $\pi:\bV\to\bV'$ denote the natural projection. Then
  $C'=\pi(C(\alpha_2,\dots,\alpha_d))=\pi(C(\alpha_2',\dots\alpha_d'))$ coincides with the image under $\pi$ of the convex hull of the set $R^+_v\cap R^+_{v'}$. On the other hand, $C'$ is a simplicial cone in $\bV'$ with generators $\pi(\alpha_2),\dots,\pi(\alpha_d)$ and similarly with generators
 $\pi(\alpha'_2),\dots,\pi(\alpha'_d)$. We choose the enumeration so that $\pi(\alpha_i)=\pi(\alpha_i')$. 
\end{proof}

\begin{lem}
  \label{lem:subset} Let $v$ be a vertex of a Matsumoto graph with basis $\{\alpha_1,\dots,\alpha_d\}$. Let $\{\alpha_1,\dots,\alpha_k\}$ be a subset and $\mathbb W$ be the span of $\{\alpha_1,\dots,\alpha_k\}$. Then $R\cap \mathbb W$ gives a geometric realization of a full subgraph $\Gamma'$ of degree $k$. 
\end{lem}
\begin{proof}
We define $\Gamma'$ as the connected component of $v$ of the subgraph
$\Gamma''$ of $\Gamma$ having the same vertices as $\Gamma$ and the edges
$e\in E$ such that $\alpha_{e,1}\in\bW$.

The vertex $v\in\Gamma'$ has $k$ incoming edges marked by
$\{\alpha_1,\dots,\alpha_k\}$.
Lemma~\ref{lem:color} implies that $\Gamma'$ is the Matsumoto graph of degree $k$ with the set of roots
  $R'\subset R\cap\mathbb W$. We will show that $R'=R\cap\mathbb W$.

  For any vertex $v'\in\Gamma'$ we denote by $\alpha_1',\dots,\alpha'_k$ the basis at $v'$.
Note first that for any $v'\in\Gamma'$ we have $R^+_{v'}\cap \mathbb W\subset C(\alpha_1',\dots,\alpha_k')$ again by Lemma \ref{lem:color}.
  Next by induction on the cardinality of $(R^+_{v'}\cap\mathbb W)\setminus (R^+_{u}\cap\mathbb W)$ we can prove that for any $u\in \Gamma$ there exists
  $v'\in \Gamma'$ such that $(R^+_{v'}\cap\mathbb W)\subset (R^+_{u}\cap\mathbb W)$.

  Now let $\beta\in R\cap \mathbb W$. There exists $u\in\Gamma$ such that $\beta$
  is an element of the basis at $u$. Then either $\beta$ or $-\beta$ lies in $R^+_{v'}\cap\mathbb W\subset C(\alpha'_1,\dots,\alpha_k')$ where $v'\in\Gamma'$ is as above. Suppose that $\beta\in R^+_{v'}\cap\mathbb W$.
  We claim that $\beta=\alpha'_i$ for some $i\leq k$. Indeed, consider a linear function $\psi$ such that $\psi(\beta)=0$ and $\psi(\gamma)>0$ for
  all other $\gamma\in R^+_u$. If $\beta\neq\alpha'_i$ for all $i\leq k$ then $\psi(\beta)>0$ and we get a contradiction.
  Similarly, if $-\beta\in R^+_{v'}\cap\mathbb W$ we can prove that $-\beta=\alpha'_i$ by the same method. The proof is complete.
  \end{proof}

\begin{notation}
\label{not:subgraph}
The graph $\Gamma'$ in the above construction will be denoted
$\Gamma(v,\alpha_1,\ldots,\alpha_k)$.
\end{notation}

\subsection{The case $d=2$}
\subsubsection{}
Here are all connected graphs of degree two.
\begin{itemize}
\item[1.] A polygon. Here all edges are compact. If an $n$-gon
admits a geometric realization,  $n$ has to be even by Lemma~\ref{lem:short}. A polygon with an even number of edges
is the Cayley graph of a dihedral group, so it admits a realization,
see~\ref{ss:coxetergroups}.
\item[2.] A graph with an infinite number of compact edges. This is 
the Cayley graph of the infinite dihedral group, so it admits a realization, see~\ref{ss:coxetergroups}.
\item[3.] A graph with one infinite edge and an infinite number of compact edges.
\item[4.] A graph with two infinite edges and a finite number of compact edges.
\end{itemize}
A geometric realization for the cases 3,4 can be easily found.
Note that in all Matsumoto cases there is a canonical coloring of the
edges so that any two adjacent edges have different colors.

\subsection{Braid relation and Matsumoto theorem}
\subsubsection{Braid relation}

Fix a vertex  $v\in\Gamma$. Choose two edges adjacent to $v$,
one incoming and one outcoming. Let $\alpha$ and $\beta$ be the corresponding roots.  There is a unique continuation of the path $(\alpha,\beta)$, potentially in both directions, so that
the roots assigned to the new edges belong to $\Span_\R(\alpha,\beta)$.
This yields a rank 2 Matsumoto subgraph $\Gamma(v,\alpha,\beta)$ of 
$\Gamma$ spanned by $\alpha$ and $\beta$. 

We denote by $\cP=\cP(\Gamma)$ the set of paths in $\Gamma$. We define
the braid relation on $\cP$ as the minimal equivalence relation 
satisfying the following two properties. 

\begin{itemize}

\item[1.] For any rank two Matsumoto  subgraph 
$\Gamma(v,\alpha,\beta)$ presented by a $2m$-gon and colored by $X=\{x,y\}$, the length $m$
paths 
$$
xyx\ldots\mathrm{ and }\ yxy\ldots
$$
starting at $v$ (and having the same end) are equivalent.
\item[2.] The equivalence relation is closed under the concatenation of paths.
\end{itemize}
Note that equivalent paths have the same ends and the same length.

The main result of this note is the following generalization of
Matsumoto theorem~\cite{M}.

\begin{thm}
\label{thm:matsumoto}
Let $\Gamma$ be a Matsumoto graph. Then any two shortest
paths in $\Gamma$ from $v$ to $v'$ are braid equivalent.
\end{thm}
The proof is given in Section~\ref{sec:proof}.

\section{Examples}
\label{sec:exm}

In this section we present two examples of Matsumoto graphs: the
Cayley graphs of Coxeter groups and the skeleta of admissible components 
of the root groupoid, see~\cite{GHS}, Section 5.

\subsection{Coxeter groups}
\label{ss:coxetergroups}

The Matsumoto structure on the Cayley graph of a general Coxeter group 
$(G,X)$, based on its geometric representation~\cite{H}, 5.3, is described below in \ref{sss:coxeter}. We present here an even more classical construction that makes sense for Weyl groups.

\subsubsection{Weyl group} Let $\fg$ be a semisimple Lie algebra
with a Cartan subalgebra $\fh$, the root system $\Delta$ and 
the set of simple roots $\alpha_1,\ldots,\alpha_n$. The Weyl group
$W$ is generated by the simple reflections $s_i=s_{\alpha_i}$.

We define $\Gamma$ as the Cayley graph of $W$ with respect to the set
$S$ of simple reflections, so that the edges (they are all compact) are
of the form $w\to ws$, $s\in S$. The arrow $w\to ws_i$ is marked with the root $\bR_{\geq 0}\cdot w(\alpha_i)$.

The same construction of Matsumoto structure generalizes to Weyl groups
of Kac-Moody Lie algebras, as well as for root Lie 
superalgebras~\cite{GHS}.

Note that the Matsumoto structures corresponding to the classical root
systems $B_n$ and $C_n$ are isomorphic, as in our formalism roots are
rays and not vectors.

\subsubsection{General Coxeter group}
\label{sss:coxeter}
Let $(G,X)$ be a Coxeter group with the generating set $X$. We define
$\Gamma$ as the Cayley graph of $G$ with respect to the set $X$ of 
generators. All edges in $\Gamma$ are compact and are colored by the
set of generators $X$. The geometric representation 
$\sigma:G\to GL(\bV)$, \cite{H}, 5.3, where
$\bV=\Span_\bR\{\alpha_x, x\in X\}$, yields a geometric realization of 
$\Gamma$ as follows. It assigns to an arrow $x:g\to gx$ the root
$\R_{\geq 0}\cdot\sigma_g(\alpha_x)\in \bV$. The set $R^+_g$ is spanned by 
$-\sigma_g(\alpha_x)$, $x\in X$. The axioms 1--3 of the geometric realization are verified immediately. The Axiom 4 follows from
\cite{H}, 5.6.

\subsection{Skeleta}
The application to skeleta is the {\sl raison-d'\^etre} of the present note.
Let $X$ be a finite set and let $v=(\fh,a:X\to\fh,b:X\to\fh^*,p:X\to\Z_2)$
be an admissible root datum, see~\cite{GHS}, Section 2.
Let $\Sk(v)$ be the connected component of $v$ in the skeleton 
of the root groupoid $\cR$, see~\cite{GHS}, Sect.~3,~4.
The graph $\Gamma$ assigned 
to it is a regular graph with the edges colored by $X$ having the
same vertices as $\Sk(v)$. The compact edges of $\Gamma$ 
are the edges of $\Sk(v)$. The 
infinite edges of $\Gamma$ are described by the nonreflectable pairs 
$(v',x)$, with $v'\in\Sk(v)$ and $x\in X$. The geometric realization of
$\Gamma$ assigns to each arrow $r_x:v\to v'$ the real root $b_{v'}(x)$
and to an infinite edge defined by the nonreflectable pair $(v',x)$
the real nonreflectable root $b_{v'}(x)$. In this realization the roots
of $\Gamma$ are precisely the real roots of the component of $v\in\cR$.

It is proven in \cite{GHS} that $\Sk(v)$ is a Coxeter graph. Now 
Theorem~\ref{thm:matsumoto} implies that the analog of Matsumoto theorem holds of $\Sk(v)$.
Note that Theorem~\ref{thm:matsumoto}  gives another proof of coxeterity of $\Sk(v)$.

\section{Proof of Theorem~\ref{thm:matsumoto}}
\label{sec:proof}

We will prove the theorem by induction in the length $n$ of the shortest paths. 

\subsection{}
The cases $n=0$ are obvious as a Matsumoto graph does not have 
loops and multiple edges. 

\subsection{}
Assume the result is proven for shortest paths of length $\leq n-1$.
Let
$$
\underline\alpha:v=v_0\xrightarrow{\alpha_1}\ldots\xrightarrow{\alpha_n}v_n=v'
$$
and
$$
\underline\beta: v=v'_0\xrightarrow{\beta_1}\ldots\xrightarrow{\beta_n}v'_n=v'
$$
be a pair of length $n$ shortest paths leading from $v$ to $v'$.
By Lemma~\ref{lem:short} the sets of $\alpha_i$ and of $\beta_j$ coincide.
Thus, $\beta_1=\alpha_i$ for some $i$. 

\subsection{}Assume that $i<n$. Consider the
path 
$$v'_1\xrightarrow{-\beta_1}v'_0=v\xrightarrow{\alpha_1}\ldots
\xrightarrow{\alpha_i}v_i.
$$
 Obviously $d(v'_1,v_i)=i-1$. Choose
a path $\gamma$ of length $i-1$ connecting $v'_1$ with $v_i$. Then, by the 
inductive hypothesis, the path $\alpha$ is braid equivalent to the 
concatenation of $\beta_1$, $\gamma$, and the segment of $\alpha$
connecting $v_i$ with $v_n=v'$. In the same manner the above concatenation
is equivalent to the path $\beta$.  This proves that in this case
the paths $\alpha$ and $\beta$ are braid equivalent. Therefore, we may assume that $i=n$, that is $\beta_1=\alpha_n$.
\subsection{}
In the same way, we can assume that $\alpha_1=\beta_n$.
\subsection{}
It remains to deal with the case when $\alpha_1=\beta_n$ and $\alpha_n=
\beta_1$. One has $\alpha_1,\beta_1\in R^+_{v'}\cap R^-_v$. This means that 
$$\Gamma(v,-\beta_1,\alpha_1)=C(\alpha_1,\beta_1)\cap R\subset R^+_{v'}\cap R^-_v.$$
The roots assigned to infinite edges are everywhere positive, 
the intersection $R^+_{v'}\cap R^-_v$ is finite, so
$\Gamma(v,-\beta_1,\alpha_1)$ is a polygon with an even number of edges.
Denote by $v''$ the vertex opposite to $v$. One has 
$$
d(v,v')=d(v,v'')+d(v'',v').
$$
Denote by $\gamma$  a shortest path connecting $v''$ with $v'$.
Denote by $\gamma'$, $\gamma''$ two equivalent halves of
$\Gamma(v,-\beta_1,\alpha_1)$
 connecting
$v$ to $v''$. Then the concatenation of $\gamma'$ with $\gamma$ is equivalent to $\alpha$ by the inductive hypothesis, as they start with the same root $\alpha_1$, and similarly the concatenation of $\gamma''$ with 
$\gamma$ is equivalent to $\beta$.

This proves the theorem.
\qed

\section{Complements}

\subsection{Edge coloring}

We will now prove the existence of a consistent coloring of the set of edges $E$ of a Matsumoto graph $\Gamma$. 

\begin{prp}
\label{prp:coloring}
Let $\Gamma$ be a graph of degree $d$ endowed with a geometric realization. Let $|X|=d$. There exists a coloring $E\to X$ of the edges of $\Gamma$ consistent with the isomorphisms defined by Lemma~\ref{lem:color}
for each arrow $v\to v'$. The coloring is unique up to automorphism of $X$.
\end{prp}
\begin{proof}
Let us identify $X$ with the set of generators of $R^+_v$.
Any closed path $\gamma:v\to v$ in $\Gamma$ defines an automorphism
$\theta_\gamma:X\to X$. It is enough to prove that $\theta_\gamma=\id_X$ for any $\gamma$.
Theorem ~\ref{thm:matsumoto} implies that if $\gamma',\gamma'':v\to v'$
are two shortest paths, $\theta_{\gamma'\circ\gamma^{\prime\prime-1}}=\id_X$.

Let us assume, to the contrary, that there exists a closed path 
$\gamma:v\to v$ for which $\theta_\gamma\ne\id_X$. Choose a shortest 
such $\gamma$, of length $2m$. For any root $\alpha$ in $\gamma$ the root
$-\alpha$ should also appear. Let us choose a closest pair of edges
is $\gamma$ with assigned roots $\alpha$ and $-\alpha$.

It is easy to see that if $\alpha$ and $-\alpha$ are not precisely opposite to each other, one can produce a shorter counterexample.

In the remaining case when all pairs $\alpha$ and $-\alpha$ are opposite
to each other, the loop $\gamma$ can necessarily be presented
as  $\gamma'\circ\gamma^{\prime\prime-1}$ where $\gamma'$ and $\gamma''$ are two shortest paths with the same ends. This proves the claim.
\end{proof}

\begin{rem}
Note that in the examples of graphs with a geometric realization presented
in Section~\ref{sec:exm}, the coloring provided by Proposition~\ref{prp:coloring} coincides with the apriori coloring
given, in the first case by the set $S$ of generator, and in the case of skeleta --- by the set $X$.
\end{rem}

\subsection{Exchange condition}

In this subsection we discuss a partial exchange condition that holds
in the context of skeleta \cite{GHS}, 4.2.5. The edges of a
skeleton $\Sk(v)$ are {\sl reflexions} that can be {\sl anisotropic}
or {\sl isotropic}, see definitions in \cite{GHS}, 4.1.1.
The following result shows that, in the context of skeleta, the exchange condition holds for anisotropic reflexions.
\begin{prp}
Let 
$$
v_0\stackrel{r_{x_1}}{\to}\ldots\stackrel{r_{x_n}}{\to}v_n
$$
be a sequence of reflexions that is a shortest path from $v_0$ to $v_n$ and let $r_y:v'_0\to v_0$ be an anisotropic reflexion such that the
path $v_0'\to v_0\to\ldots\to v_n$ is not shortest. Then there exists 
$j$ so that the composition
$$ v'_0\stackrel{r_y}{\to} v_0\to\ldots\stackrel{r_{x_i}}{\to}v_j$$ 
coincides with the composition
$$
v'_0\stackrel{r_{x_1}}{\to}v'_1\to\ldots
\stackrel{r_{x_{j-1}}}{\to}v'_{j-1}=v_j.
$$
\end{prp}
\begin{proof}
By the assumptions, $\alpha=b_{v_0}(y)=b_{v_j}(x_j)$ for some  $j$.
The path $v'_0\to\ldots\to v'_{j-1}$ exists as the namesake of the path
$v_0\to\ldots\to v_{j-1}$, see~\cite{GHS}, 4.3.6, 4.3.7. The equality
$v'_{j-1}=v_j$ then follows from 4.3.7.
\end{proof}

\subsection{Geometric properties of roots}

\begin{lem}
    \label{lem:limit_points} Let $\Gamma$ be a Matsumoto graph with set of roots $R$. Consider $R$ as a set of points on the unit sphere in $\bV$.
    If $\alpha,-\alpha\in R$  (that is, if $\alpha$ is 
    the root of a compact edge) then $\alpha$ is an isolated point of $R$ in the usual topology on the unit sphere.
  \end{lem}
\begin{proof} 
Let $\alpha$ mark an arrow $v'\to v$, so that $\alpha=\alpha_1,\ldots,\alpha_d$ form the basis of roots at $v$. If there is a sequence of roots
$\beta_k\in R^+_v$ converging to $\alpha$, then $\alpha$ belongs to the closure of $R^+_v\setminus\alpha=R^+_v\cap R^+_{v'}\subset C(\alpha_2,\ldots,\alpha_d)$ which is impossible. For the same reason a sequence 
of $\beta_k\in R^-_v$ cannot possibly converge to $\alpha$:
then $-\beta_k\in C(\alpha_2,\ldots,\alpha_d)$ would converge to $-\alpha$. This proves the result.
\end{proof}

  Note that a noninvertible root $\alpha$ may be a limit point in $R$ as we will see in~\ref{exa:nonisolated} below. 

\subsection{Dual picture}
\label{ss:dual}
The notion of Matsumoto graph can be equivalently expressed in
the dual picture, in terms of convex geometry. We present below
this equivalent language.

\hbox{}

\subsubsection{}

Let $\Gamma$ be a Matsumoto graph with the set of roots $R$ in a vector space $\bV$. For every $\alpha\in R$ define a hyperplane $H_{\alpha}$ and a half-space $U_{\alpha}$
  in the dual space $\bV^*$ by
  $$H_{\alpha}=\{\xi\in\bV^*\mid \langle \xi,\alpha \rangle=0\},\quad U_{\alpha}=\{\xi\in\bV^*\mid \langle \xi,\alpha \rangle\geq 0\}.$$
  For a vertex $v\in\Gamma$ with assigned basis $\alpha_1,\dots,\alpha_d$ we define a chamber $C^\vee_v$ to be the intersection $\cap U_{\alpha_i}$.
Then $C^\vee_v$ is a simplicial cone in $\bV^*$.
\begin{lem}
\label{lem:convexity} Let
$D=\cup_{v\in\Gamma} C^\vee_v$ and 
$D'$ be the set of all $\xi\in\bV^*$ such that $\langle\xi,\alpha \rangle\geq 0$ for all but finitely many invertible $\alpha\in R_v^+$. Then $D=D'$. In particular, $D$ is conical convex set. 
\end{lem}
\begin{proof} Obviously $D\subset D'$. Let us show that $D=D'$. Suppose that $\xi\in D'$ and $M$ be the set of all invertible 
$\alpha\in R_v^+$ such that $\langle \xi,\alpha \rangle<0$. If $M\neq\emptyset$ then there is an invertible root $\alpha$ in the basis at $v$ such that $\alpha\in M$.
 Consider the arrow $v'\to v$ marked by $\alpha$. Then $\langle\xi,\beta \rangle\geq 0$ for all $\beta\in R_{v'}^+\setminus M'$ with
  $M'=M\setminus\{\alpha\}$. Using induction on $|M|$ we can prove existence of a vertex $u$ such that $\xi$ is nonnegative on all roots of $R^+_u$.
  This shows that $\xi\in C_u\subset D$.
  \end{proof}

\subsubsection{} Vice versa, here is a way to reconstruct a Matsumoto graph from a convex conical set in the dual space $\bV^*$.
 
Let $D$ be a conical convex set in $\bV^*$ and $R$ a collection of rays in $\bV$ so that each $H_\alpha$ intersects $D$ nontrivially. We assume that all connected components of \ $D\setminus(\cup H_\alpha)$ are interiors
of simplicial cones $C^\vee_v$ and $D=\cup C^\vee_v$. One can now construct a Matsumoto graph $\Gamma$ whose vertices correspond to the chambers and edges correspond to $d-1$-dimensional faces of chambers with adjacency of
chambers defined in the obvious way. This construction gives a dual geometric realization of a Matsumoto graph.

\begin{exm}
\label{exa:nonisolated}

The convex conic set $D$ is uniquely determined by its
intersection with the unit sphere in $\bV^*$. Moreover, if
$D\setminus\{0\}$ lies in some open half-space
$$
\{\xi\in\bV^*|\ \langle\xi,\alpha\rangle>0\},
$$
one can use instead the intersection of $D$ with the affine hyperplane
$H=\{\xi\in\bV^*|\ \langle\xi,\alpha\rangle=1\}$.

Let us give one example in the case $\dim H=2$.
  Consider a triangle $A,B,C$ in $H$, let $M_0$ be the midpoint $BC$, $M_1$ the midpoint of $AC$, $M_2$ the midpoint of $M_1C$, and then inductively $M_n$ is defined to be the midpoint of $M_{n-1} C$. This configuration gives a dual realization of a Matsumoto graph with chambers
  $ABM_0, AM_0M_1,\dots, AM_nM_{n+1}\ (n>0),\dots$. 
   The roots are dual to the lines $AB$, $BM_0$, $M_0C$, $M_0M_i$ for all $i>0$, $AM_1$, $M_nM_{n+1}$. The noninvertible roots are
  $AB$, $BM_0$, $M_0C$, $AM_1$ and $M_nM_{n+1}$ for $n>0$. Note that $M_0C$ is the limit of the set $\{M_0M_n\}_{n\geq 1}$.  

\end{exm}

\end{document}